\documentclass[final]{siamltex}

\usepackage[parfill]{parskip}    % Activate to begin paragraphs with an empty line rather than an indent
\usepackage{amsmath,amssymb,latexsym,enumerate, mathrsfs}
\usepackage{hyperref}
\usepackage{array}
\usepackage{times}

\usepackage{epstopdf}
\usepackage{subfig}
\usepackage{graphicx}

\usepackage{moreverb}
\usepackage{marvosym}
\usepackage{color,soul}
\definecolor{lightblue}{rgb}{.90,.95,1}
\sethlcolor{yellow}

\DeclareMathOperator*{\argmax}{arg\,max}
\newtheorem{remark}{Remark}

\newtheorem{assumption}{\bf Assumption}

\title{On the minimum exit rate for a diffusion process pertaining to a chain of distributed control systems with random perturbations
\thanks{Received by the editors September 08, 2014. This work was supported in part by the National Science Foundation under Grant No. CNS-1035655. The first author acknowledges support from the Department of Electrical Engineering, University of Notre Dame.}}

\author{Getachew K. Befekadu\footnotemark[2] 
\and Panos J. Antsaklis\footnotemark[2]}

\begin{document}
\maketitle

\renewcommand{\thefootnote}{\fnsymbol{footnote}}

\footnotetext[2]{Department of Electrical Engineering, University of Notre Dame, Notre Dame, IN 46556, USA ({\tt gbefekadu1@nd.edu, antsaklis.1@nd.edu}).}

\renewcommand{\thefootnote}{\arabic{footnote}}

\begin{abstract}
In this paper, we consider the problem of minimizing the exit rate with which a diffusion process pertaining to a chain of distributed control systems, with random perturbations, exits from a given bounded open domain. In particular, we consider a chain of distributed control systems that are formed by $n$ subsystems (with $n \ge 2$), where the random perturbation enters only in the first subsystem and is then subsequently transmitted to the other subsystems. Furthermore, we assume that, for any $\ell \in \{2, \ldots, n\}$, the distributed control systems, which is formed by the first $\ell$ subsystems, satisfies an appropriate H\"{o}rmander condition. As a result of this, the diffusion process is degenerate, in the sense that the infinitesimal generator associated with it is a degenerate parabolic equation. Our interest is to establish a connection between the minimum exit rate with which the diffusion process exits from the given domain and the principal eigenvalue for the infinitesimal generator with zero boundary conditions. Such a connection allows us to derive a family of Hamilton-Jacobi-Bellman equations for which we provide a verification theorem that shows the validity of the corresponding optimal control problems. Finally, we provide an estimate on the attainable exit probability of the diffusion process with respect to a set of admissible (optimal) Markov controls for the optimal control problems.
\end{abstract}



\pagestyle{myheadings}
\thispagestyle{plain}
\markboth{G. K. BEFEKADU AND P. J. ANTSAKLIS}{ON THE MINIMUM EXIT RATE FOR A DIFFUSION PROCESS}

\section{Introduction}	 \label{S1}

In this paper, we consider the problem of minimizing the exit rate with which the diffusion process $x^{ i}(t)$, for $i = 1, 2, \ldots, n$, exits from a given bounded open domain pertaining to the following distributed control systems with random perturbations (see Fig.~\ref{Fig-DCS})\footnote{This work is, in some sense, a continuation of our previous paper \cite{BefAn14}.}
\begin{align}
\left.\begin{array}{l}
d x^{1}(t) = m_1\bigl(x^{1}(t), u_1(t)\bigr) dt + \sigma\bigl(x^{1}(t))dW(t) \\
d x^{2}(t) = m_2\bigl(x^{1}(t), x^{2}(t), u_2(t)\bigr) dt  \\
 \quad\quad\quad~ \vdots  \\
d x^{i}(t) = m_{i}\bigl(x^{1}(t), x^{2}(t), \ldots, x^{i}(t), u_{i}(t)\bigr) dt  \\
 \quad\quad\quad~ \vdots  \\
d x^{n}(t) = m_n\bigl(x^{1}(t), x^{2}(t), \ldots, x^{n}(t), u_{n}(t)\bigr) dt\\
 x^{1}(0)=x_0^{1}, \,\, x^{2}(0)=x_0^{2}, \,\, \ldots, \,\, x^{n}(0)=x_0^{n}, \,\, t \ge 0
\end{array}\right\}  \label{Eq1} 
\end{align}
where
\begin{itemize}
\item[-] $x^{i}(\cdot)$ is an $\mathbb{R}^{d}$-valued diffusion process that corresponds to the $i$th-subsystem (with $i \in \{1,2, \ldots, n\}$ and $n \ge 2$),
\item[-] the functions $m_i \colon \mathbb{R}^{i \times d} \times \mathcal{U}_i \rightarrow \mathbb{R}^{d}$ are uniformly Lipschitz, with bounded first derivatives,
\item[-] $\sigma \colon \mathbb{R}^{d} \rightarrow \mathbb{R}^{d \times m}$ is Lipschitz with the least eigenvalue of $\sigma(\cdot)\,\sigma^T(\cdot)$ uniformly bounded away from zero, i.e., 
\begin{align*}
 \sigma(x)\,\sigma^T(x) \ge \lambda I_{d \times d} , \quad \forall x \in \mathbb{R}^{d},
\end{align*}
for some $\lambda > 0$,
\item[-] $W(\cdot)$ (with $W(0) = 0$) is a $d$-dimensional standard Wiener process,
\item[-] $u_i(\cdot)$ is a $\,\mathcal{U}_i$-valued measurable control process to the $i$th-subsystem (i.e., an admissible control from the measurable set $\mathcal{U}_i\subset \mathbb{R}^{r_i}$) such that for all $t > s$, $W(t)-W(s)$ is independent of $u_i(\nu)$ for $\nu \le s$ (nonanticipativity condition) and
\begin{align*}
\mathbb{E} \int_{0}^{t_1} \vert u_i(t)\vert^2 dt < \infty, \quad \forall t_1 \ge 0,
\end{align*}
\end{itemize}
for $i = 1, 2, \ldots, n$.

\begin{figure}[tbh]
\vspace{-3 mm}
\begin{center}
\subfloat{\includegraphics[width=36mm]{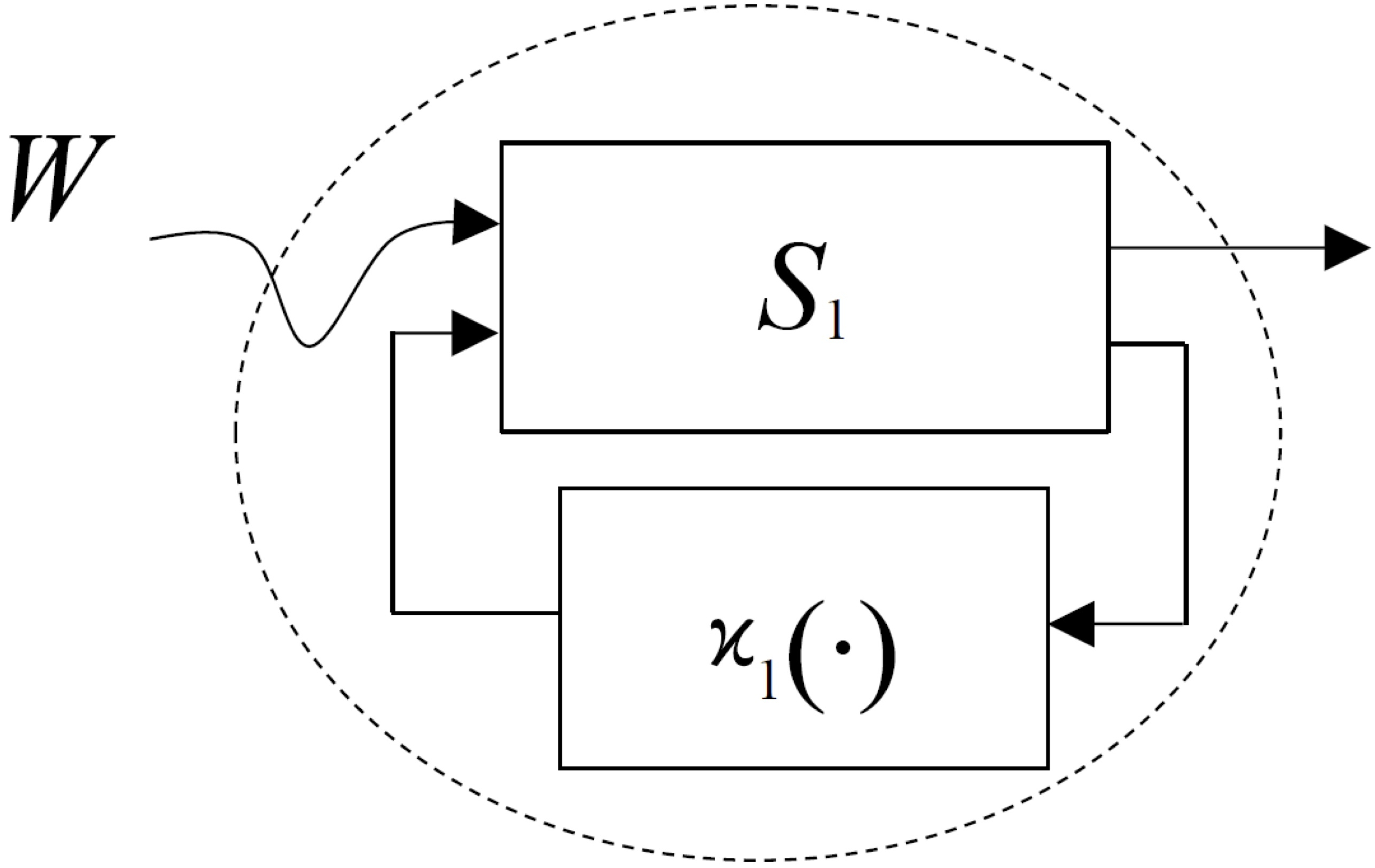}}
\subfloat{\includegraphics[width=46mm]{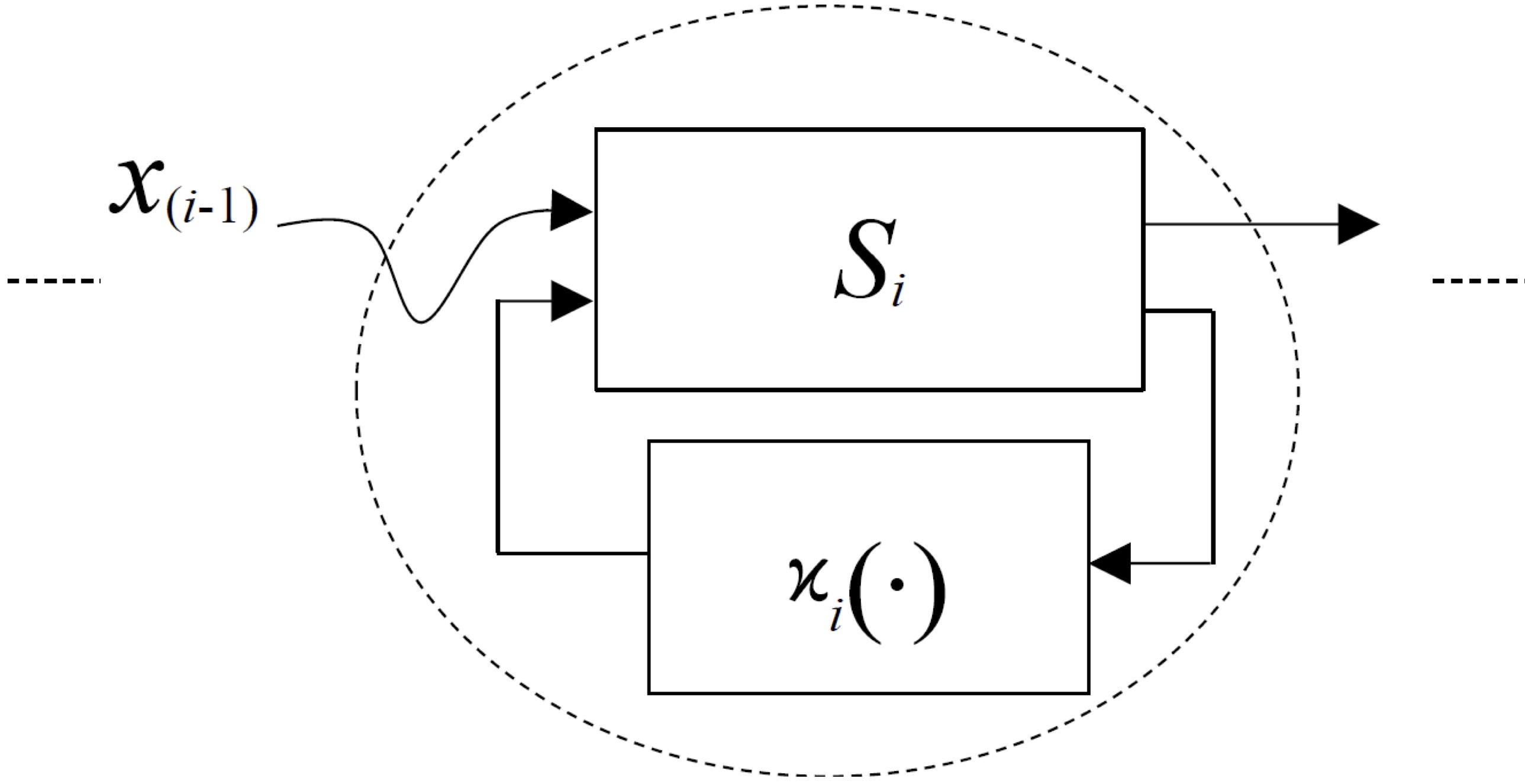}}
\subfloat{\includegraphics[width=37mm]{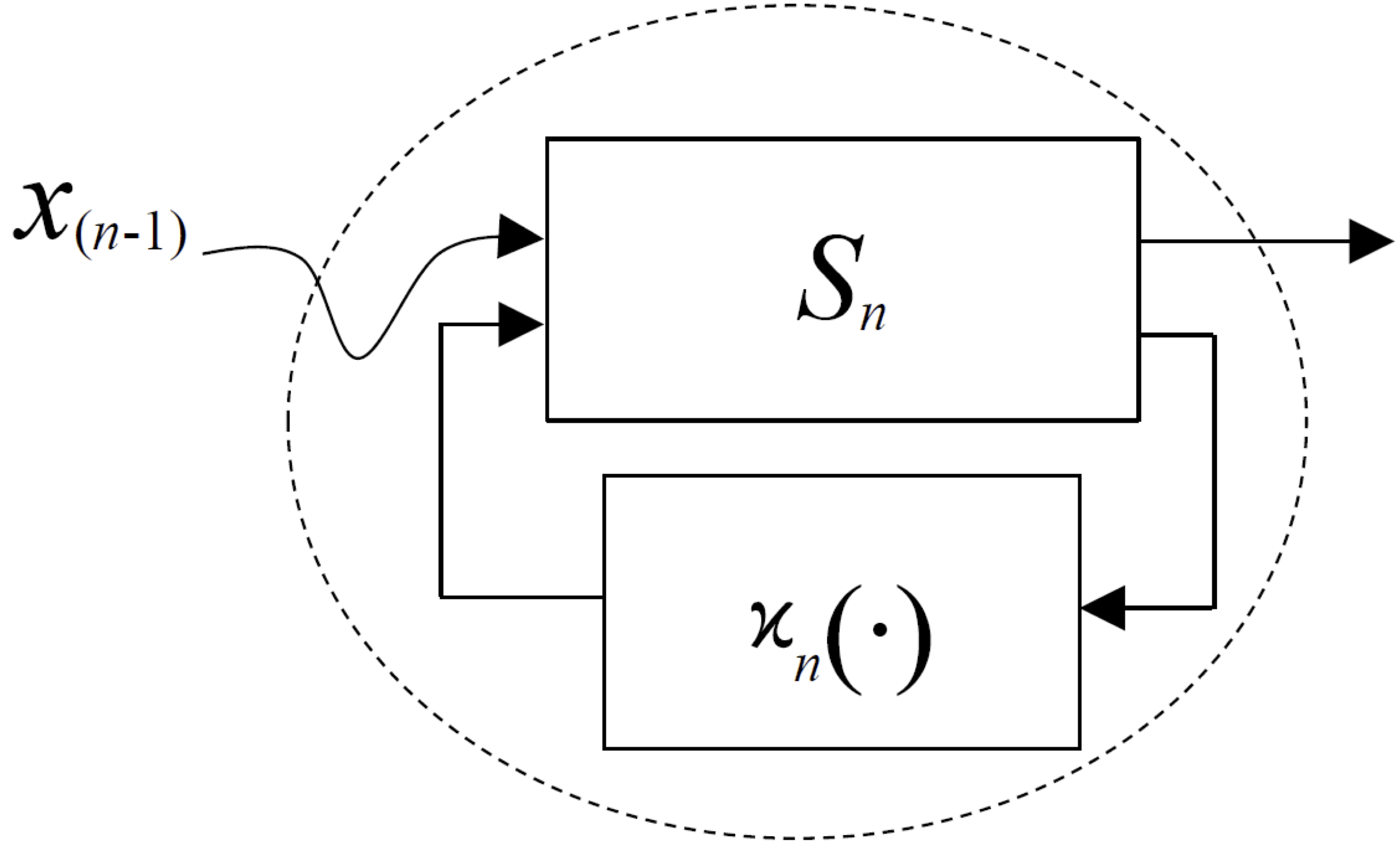}}\\
{$ \begin{array}{r@{\ }c@{\ }l}\\
\text{where} & \\
&S_1:  \, \, d x^{1}(t) = m_1\bigl(x^{1}(t), u_1(t)\bigr) dt + \sigma\bigl(x^{1}(t))dW(t), \\
&S_i: \,\, d x^{i}(t) = m_{i}\bigl(x^{1}(t), x^{2}(t), \ldots, x^{i}(t), u_{i}(t)\bigr) dt, ~ i = 2, 3, \ldots n,\\
     & \quad \quad u_j = \kappa_j(x^{j}), ~ j =1, 2,  \ldots n.
\end{array}$}
\caption{A chain of distributed control systems with random perturbations} \label{Fig-DCS}
\vspace{-3 mm}
\end{center}
\end{figure}

In what follows, we consider a particular class of admissible controls $u_{i}(\cdot) \in \mathcal{U}_{i}$, for $i = 1, 2 \ldots, n$, of the form $\kappa_{i}(x^{i}(t))$, $t \ge 0$, with a measurable map $\kappa_{i}$ from $\mathbb{R}^{d}$ to $\mathcal{U}_{i}$ and, thus, such a measurable map $\kappa_i$ is called a stationary Markov control. 

\begin{remark} \label{R1}
Note that the function $m_i$, with an admissible control $\kappa_i(x^{i})$, depends only on $x^{1}$, $x^{2}$, \ldots, $x^{i}$. For any $\ell \in \{2, \ldots, n\}$, we assume that the distributed control systems, which is formed by the first $\ell$ subsystems, satisfies an appropriate H\"{o}rmander condition (e.g., see \cite{Hor67}). Furthermore, the random perturbation has to pass through the second subsystem, the third subsystem, and so on to the $i$th-subsystem. Hence, such a chain of distributed control systems is described by an $n \times d$ dimensional diffusion process, which is degenerate in the sense that the infinitesimal generator associated with it is a degenerate parabolic equation (see also Remark~\ref{R3}). 
\end{remark}

Let $D_i \subset \mathbb{R}^{d}$, for $i =1, 2, \ldots, n$, be bounded open domains with smooth boundaries (i.e., $\partial D_i$ is a manifold of class $C^2$). Moreover, let $\Omega_{\ell}$ be the open sets that are given by
\begin{align*}
\Omega_{\ell} = D_1 \times D_2 \times \cdots \times D_{\ell}, \quad \ell = 2, 3, \ldots, n.
\end{align*}

Suppose that the distributed control systems in Equation~\eqref{Eq1} are composed with a set of admissible Markov controls $\kappa_{i}$, for $i =1, 2, \ldots, n$. Further, let $\tau^{\ell}=\tau^{{\ell}}(x^{1}, \ldots, x^{\ell})$, for $\ell \in \{2, 3, \ldots, n\}$, be the first exit-time for the diffusion process $x^{ \ell}(t)$ (which corresponds to the $\ell$th-subsystem) from the given domain $D_{\ell}$, i.e., 
\begin{align}
\tau^{\ell} = \inf \Bigl\{ t > 0 \, \bigl\vert \, x^{\ell}(t) \in \partial D_{\ell} \Bigr\}, \label{Eq2}
\end{align}
which also depends on the behavior of the solutions to the following (deterministic) distributed control systems
\begin{align}
\left.\begin{array}{l}
d \xi^{1}(t) = \hat{m}_1\bigl(\xi^{1}(t)\bigr) dt \\
d \xi^{j}(t) = \hat{m}_{j}\bigl(\xi^{1}(t), \xi^{2}(t), \ldots, \xi^{j}(t)\bigr) dt,  \,\,  t \ge 0, \,\, j = 2, \ldots, \ell
\end{array}\right\}, \label{Eq3}
\end{align}
where
\begin{align*}
\hat{m}_{i} \bigl(\xi^{1}(t), \ldots, \xi^{i}(t)\bigr) \triangleq m_{i}\bigl(\xi^{1}(t), \ldots, \xi^{i}(t), \kappa_{i}(\xi^{i}(t))\bigr),\,\, i = 1, 2, \ldots, \ell.
\end{align*} 

\begin{remark} \label{R2}
In this paper, we assume that the origin is a (unique) globally asymptotically stable equilibrium point for the above (deterministic) distributed control systems.
\end{remark}

Notice that the infinitesimal generator for the diffusion processes $\bigl(x^{1}(t), x^{2}(t), \ldots, x^{\ell}(t)\bigr)$ (when $u_i(t) = \kappa_{i}(x^{i}(t))$, $\forall t \ge 0$, for $i = 1, 2, \ldots, \ell$) is given by
\begin{eqnarray}
\mathcal{L}_{\kappa}^{\ell} \bigl(\cdot\bigr) \bigl(x^{1}, \ldots, x^{\ell}\bigr) = \sum\nolimits_{i=1}^{\ell} \bigl \langle \bigtriangledown(\cdot), \hat{m}_i \bigl(x^{1}, \ldots, x^{i}\bigr) \bigr\rangle + \frac{1}{2} \operatorname{tr}\Bigl \{\sigma(x^1)\sigma^T(x^1)\bigtriangledown^2(\cdot) \Bigr\} \label{Eq4}
\end{eqnarray}
with $a(x^{1})=\sigma(x^{1})\,\sigma^T(x^{1})$ and 
\begin{align*}
\hat{m}_{i}\bigl(x^{1}(t), \ldots, x^{i}(t)\bigr) = m_{i}\bigl(x^{1}(t), \ldots, x^{i}(t), \kappa_{i}(x^{i}(t))\bigr), \,\, i = 1, 2, \ldots, \ell.
\end{align*}

In this paper, we assume that the following statements hold for the chain of distributed control systems in Equation~\eqref{Eq1}.

\begin{assumption} \label{AS1} ~\\\vspace{-3mm}
\begin{enumerate} [(a)]
\item The infinitesimal generator in Equation~\eqref{Eq4} is hypoelliptic in $C^{\infty}(\Omega_{\ell})$, for each $\ell =2, 3, \ldots, n$ (cf. Remark~\ref{R3}).
\item $\bigl \langle \hat{m}_{\ell}\bigl(x^{1}, x^{2}, \ldots, y \bigr), \alpha(y) \bigr\rangle > 0$, for each $\ell = 2, 3, \ldots, n$, where $\alpha(y)$ is a unit outward normal vector to $\partial D_{\ell}$ at $y \in \partial D_{\ell}$.\footnote{We remark that
\begin{align*}
\mathbb{P}_{x_{\widehat{1,\ell}}} \Bigl\{ \bigl( x^{1}(\tau^{\ell}), x^{2} (\tau^{\ell}), \ldots, x^{\ell} (\tau^{\ell})\bigr) \in \Gamma_{\ell}, \,\, \tau^{\ell} < \infty \Bigr \} =1, \,\,
 \forall \bigl(x_0^{1}, x_0^{2}, \ldots, x_0^{\ell} \bigr) \in \Omega_{\ell},
\end{align*}
where $\Gamma_{\ell}$ denotes the set of points $\bigl(x^{1}, x^{2}, \ldots, y\bigr)$, with $y \in \partial D_{\ell}$, such that
\begin{align*}
\bigl \langle \hat{m}_{\ell}\bigl(x^{1}, x^{2}, \ldots, y \bigr), \alpha(y) \bigr\rangle > 0, \quad \ell = 2, 3, \ldots, n.
\end{align*}
Moreover, if $\tau^{\ell} \le T$, with fixed $T > 0$. Then, we have $\bigl( x^{1}(\tau^{\ell}), x^{2} (\tau^{\ell}), \ldots, x^{\ell} (\tau^{\ell})\bigr) \in \Gamma_{\ell}$ almost surely (see \cite[Section~7]{StrVa72}).}
\end{enumerate}
\end{assumption}

\begin{remark}  \label{R3}
Note that, from Assumptions~\ref{AS1}(a), each matrix $\Bigr(\frac{\partial {m_{\ell}}_i}{\partial x_j^{1}}\Bigl)_{ij}$, for $\ell = 2, 3, \ldots n$, has full rank $d$ everywhere in $\Omega_{\ell}$, due to the hypoellipticity of the infinitesimal generator in Equation~\eqref{Eq4}. In general, the hypoellipticity is related to a strong accessibility property of controllable nonlinear systems that are driven by white noise (e.g., see \cite{SusJu72} concerning the controllability of nonlinear systems, which is closely related to \cite{StrVa72} and \cite{IchKu74}; and see also \cite[Section~3]{Ell73}). That is, the hypoellipticity assumption implies that the diffusion process $x^{\ell}(t)$ has a transition probability density $p_{(t)}^{\ell}\bigl((x^{1}, \ldots, x^{\ell}), \otimes_{i=1}^{\ell} d\mu_i\bigr)$, which is $C^{\infty}$ on $\mathbb{R}^{2(d \times \ell)}$, and which satisfies the forward equation ${\mathcal{L}_{\kappa}^{\ell}}^{\ast} p_{(t)}^{\ell} = 0$ as $t \rightarrow \infty$, where ${\mathcal{L}_{\kappa}^{\ell}}^{\ast}$ is adjoint of the infinitesimal generator $\mathcal{L}_{\kappa}^{\ell}$.
\end{remark}

In what follows, we consider the following nondegenerate diffusion processes $(x^{1}(t), x^{\epsilon_2, 2}(t), \\ \ldots, x^{\epsilon_\ell, \ell}(t))$, for each $\ell = 2, 3, \ldots, n$, satisfying
\begin{align}
\left.\begin{array}{l}
d x^{1}(t) = \hat{m}_1\bigl(x^{1}(t)\bigr) dt +  \sigma\bigl(x^{1}(t))dW(t) \\
d x^{\epsilon_j, j}(t) = \hat{m}_{j}\bigl(x^{1}(t), x^{\epsilon_2, 2}(t), \ldots, x^{\epsilon_j, j}(t)\bigr) dt + \sqrt{\epsilon_j} dW_j(t) dt \\
\quad \quad\quad\quad\quad\quad\quad\quad \quad\quad\quad\quad\quad\quad  0< \epsilon_j \ll 1,  \,\,\,  t \ge 0, \,\,\, j = 2, \ldots, \ell
\end{array}\right\}, \label{Eq5}
\end{align}
 where
\begin{align*}
\hat{m}_{i}\bigl(x^{1}(t), \ldots, x^{\epsilon_i, i}(t)\bigr) = m_{i}\bigl(x^{1}(t), \ldots, x^{\epsilon_i, i}(t), \kappa_{i}(x^{\epsilon_i, i}(t))\bigr), \,\, i = 1, 2, \ldots, \ell 
\end{align*}
and with an initial condition
\begin{align*}
  \bigl(x_0^{1}, x_0^{\epsilon_2, 2}, \ldots, x_0^{\epsilon_\ell, \ell}\bigr) = \bigl(x_0^{1}, x_0^{2}, \ldots, x_0^{\ell}\bigr). 
\end{align*}
Furthermore, $W_j(\cdot)$ (with $W_j(0) = 0$), for $j = 2, \ldots, n$, are $d$-dimensional standard Wiener processes and independent to $W(\cdot)$.

The infinitesimal generator, which is associated with the nondegenrate diffusion processes in Equation~\eqref{Eq5}, is given by
\begin{align}
& \mathcal{L}_{\kappa}^{\epsilon,\ell} \bigl(\cdot\bigr) \bigl(x^{1}, x^{\epsilon_2, 2}, \ldots, x^{\epsilon_{\ell}, \ell}\bigr) = \sum\nolimits_{i=1}^{\ell} \bigl \langle \bigtriangledown(\cdot), \hat{m}_i \bigl(x^{1}, \ldots, x^{\epsilon_i, i}\bigr) \bigr\rangle \notag \\ 
   & \quad \quad \quad\quad\quad\quad \quad + \frac{1}{2} \operatorname{tr}\Bigl \{\sigma(x^1)\sigma^T(x^1)\bigtriangledown^2(\cdot) \Bigr\}  + \sum\nolimits_{j=2}^{\ell} \frac{\epsilon_j}{2} \operatorname{tr}\Bigl \{ I_{d \times d}\bigtriangledown^2(\cdot) \Bigr\}, \label{Eq6}
\end{align}
for each $\ell = 2, 3, \ldots, n$.

In Section~\ref{S2}, we present our main results -- where, in Subsection~\ref{S2(1)}, we establish a connection between the minimum exit rate with which the diffusion process $x^{\epsilon_{\ell}, \ell}$ (for the nondegenerate case) exits from the domain $D_{\ell}$ and the principal eigenvalue of the infinitesimal generator $\mathcal{L}_{\kappa}^{\epsilon,\ell}$, with zero boundary conditions on $\partial D_{\ell}$. In Subsection~\ref{S2(2)}, using the formalism from Subsection~\ref{S2(1)}, we derive a family of Hamilton-Jacobi-Bellman (HJB) equations for which we provide a verification theorem that shows the validity of the corresponding optimal control problems. Finally, in Subsection~\ref{S2(2)}, we provide an estimate on the attainable exit probability of the diffusion process $x^{\ell}(t)$, from the given domain $D_{\ell}$, with respect to a set of admissible optimal Markov controls $\kappa_i^{\ast}$, for $i = 1, 2, \ldots, \ell$.

Before concluding this section, it is worth mentioning that some interesting studies on estimating density functions for dynamical systems with random perturbations have been reported in literature (for example, see \cite{Hern81} or \cite{DelM10} in the context of estimating density functions for degenerate diffusions; see \cite{She91}, \cite{Day87} or \cite[Chapters~12 and 13]{Fre76} in the context of nondegenerate diffusions; and see \cite{Kif81} or \cite{BisB11} in the context of an asymptotic behavior for the equilibrium density).

\newpage
\section{Main Results} \label{S2}
\subsection{Minimum exit rates and principal eigenvalues} \label{S2(1)}
For a fixed $\ell \in \{2, 3, \ldots, n\}$, let us consider the following eigenvalue problem
\begin{align}
\left.\begin{array}{c}
  - \mathcal{L}_{\kappa}^{\epsilon,\ell} \psi_{\kappa}^{\ell} \bigl(x^{1}, x^{\epsilon_2, 2} \ldots, x^{\epsilon_{\ell}, \ell}\bigr) = \lambda_{\kappa}^{\epsilon,\ell} \psi_{\kappa}^{\ell}\bigl(x^{1}, x^{2} \ldots, x^{\ell}\bigr) \quad \text{in} \quad \Omega_{\ell} \quad \vspace{2mm} \\
  \quad \psi_{\kappa}^{\ell}\bigl(x^{1}, x^{\epsilon_2, 2} \ldots, x^{\epsilon_{\ell}, \ell}\bigr) = 0 \quad \text{on} \quad \partial \Omega_{\ell},  \label{Eq7}
  \end{array}\right\} 
\end{align}
where the infinitesimal generator $\mathcal{L}_{\kappa}^{\epsilon,\ell}$ is given by Equation~\eqref{Eq6}. 

In this subsection, using Theorems~1.1, 1.2 and 1.4 from \cite{QuaSi08}, we provide conditions for the existence of a unique principal eigenvalue $\lambda_{\kappa}^{\epsilon,\ell} > 0$ and an eigenfunction $\psi_{\kappa}^{\ell} \in W_{loc}^{2,p} \bigl(\Omega_{\ell}\bigr) \cap C\bigl(\bar{\Omega}_{\ell}\bigr)$ pairs for the eigenvalue problem in Equation~\eqref{Eq7}, with zero boundary conditions on $\partial \Omega_{\ell}$. Note that the principal eigenvalue $\lambda_{\kappa}^{\epsilon,\ell}$ is related to the minimum exit rate with which the diffusion process $x^{\epsilon_{\ell}, \ell}(t)$ exits from the domain $D_{\ell}$ -- when the distributed control systems in Equation~\eqref{Eq5} are composed with a set of admissible Markov controls $\kappa_i$, for $i = 1, 2, \dots, \ell$.

The following proposition establishes a connection between the minimum exit rate and with that of the principal eigenvalue for the infinitesimal generator $\mathcal{L}_{\kappa}^{\epsilon,\ell}$ in Equation~\eqref{Eq6}.

\begin{proposition} \label{P1}
Suppose that a set of admissible Markov controls $\kappa_i$, for $i=1,2, \ldots n$, is given. Then, the principal eigenvalue $\lambda_{\kappa}^{\epsilon,\ell}$ for the infinitesimal generator $\mathcal{L}_{\kappa}^{\epsilon, \ell}$, with zero boundary conditions on $\partial \Omega_{\ell}$, is given by 
\begin{align}
\lambda_{\kappa}^{\epsilon,\ell} = - \limsup_{t \rightarrow \infty} \frac{1} {t} \log \mathbb{P}_{x_{\widehat{1,\ell}}}^{\epsilon} \bigl\{\tau^{\ell, \epsilon} > t \bigr\},  \label{Eq8}
\end{align}
for each $\ell = 2, 3, \ldots n$, where $\tau^{\ell,\epsilon}$ is the exit-time for the diffusion process $x^{\epsilon_{\ell}, \ell}(t)$ (which corresponds to the $\ell$th-subsystem in the nondegenerate case) from the domain $D_{\ell}$, i.e.,
\begin{align*}
\tau^{\ell,\epsilon} = \inf \Bigl\{ t > 0 \, \bigl\vert \, x^{\epsilon_{\ell}, \ell}(t) \in \partial D_{\ell} \Bigr\},
\end{align*}
and the probability $\mathbb{P}_{x_{\widehat{1,\ell}}}^{\epsilon}\bigl\{\cdot\bigr\}$ in Equation~\eqref{Eq8} is further conditioned on the initial condition $\bigl(x_0^{1}, x_0^{\epsilon_2, 2}, \ldots, x_0^{\epsilon_{\ell}, \ell}\bigr) \in \Omega_{\ell}$ as well as on the admissible Markov controls $\kappa_i$, for $i=1,2, \ldots \ell$.
\end{proposition}

\begin{proof}
Fix $\ell \in \{ 2, 3, \ldots, n\}$, and for $\delta > 0$, let $D_{\ell}^{\delta} \subset D_{\ell}$ (with $D_{\ell}^{\delta} \cup \partial D_{\ell}^{\delta} \subset D_{\ell}$) be a family of bounded domains with smooth boundaries, increasing to $D_{\ell}$ as $\delta \rightarrow 0$. Let 
\begin{align*}
\tau_{\delta}^{\ell,\epsilon} = \inf \Bigl\{ t > 0 \, \bigl\vert \, x^{\epsilon_{\ell}, \ell}(t) \in \partial D_{\ell}^{\delta} \Bigr\}.
\end{align*}
Then, applying Krylov's extension of the It\^{o}'s formula valid for continuous functions from $W_{loc}^{2,p} \bigl(\Omega_{\ell}\bigr)$, with $p \ge 2$ (e.g., see \cite[Chapter~II]{Bor89}) and the optional sampling theorem
\begin{align*}
 \psi_{\kappa}^{\ell} \Bigl(\hat{x}^{1}, \hat{x}^{2}, \ldots, \hat{x}^{\ell}\Bigr) &= \mathbb{E}_{x_{\widehat{1,\ell}}}^{\epsilon} \Biggl\{\exp\bigl( \lambda_{\kappa}^{\epsilon,\ell} (t \wedge \tau_{\delta}^{\ell,\epsilon}) \bigr) \notag \\
  & \quad\quad\quad  \times \psi_{\kappa}^{\ell} \Bigl(x^{1}(t \wedge \tau_{\delta}^{\ell,\epsilon}), x^{\epsilon_2, 2}(t \wedge \tau_{\delta}^{\ell,\epsilon}), \ldots, x^{\epsilon_{\ell}, \ell}(t \wedge \tau_{\delta}^{\ell,\epsilon})\Bigr) \Biggr\}
\end{align*}
for some $\bigl(\hat{x}^{1}, \hat{x}^{2}, \ldots, \hat{x}^{\ell}\bigr) \in \Omega_{\ell}$.
Letting $\delta \rightarrow 0$, then we have $\tau_{\delta}^{\ell,\epsilon} \rightarrow \tau^{\ell,\epsilon}$, almost surely, and
\begin{align*}
 \psi_{\kappa}^{\ell} \Bigl(\hat{x}^{1}, \hat{x}^{2}, \ldots, \hat{x}^{\ell}\Bigr) 
 &= \mathbb{E}_{x_{\widehat{1,\ell}}}^{\epsilon} \Biggl\{\exp\bigl( \lambda_{\kappa}^{\epsilon,\ell} (t \wedge \tau_{\delta}^{\ell,\epsilon}) \bigr) \\
 & \quad\quad\quad  \times \psi_{\kappa}^{\ell} \Bigl(x^{1}(t \wedge \tau_{\delta}^{\ell,\epsilon}), x^{\epsilon_2, 2}(t \wedge \tau_{\delta}^{\ell,\epsilon}), \ldots, x^{\epsilon_{\ell}, \ell}(t \wedge \tau_{\delta}^{\ell,\epsilon})\Bigr) \Biggr\} \\
&= \mathbb{E}_{x_{\widehat{1,\ell}}}^{\epsilon} \Biggl\{\exp\bigl( \lambda_{\kappa}^{\epsilon,\ell} t \bigr) \,\psi_{\kappa}^{\ell} \Bigl(x^{1}(t), x^{\epsilon_2, 2}(t), \ldots, x^{\epsilon_{\ell}, \ell}(t)\Bigr) \mathbf{1}\Bigl\{\tau_{\delta}^{\ell,\epsilon} > t\Bigr\} \Biggr\} \\
  & \le  \Bigl\Vert \psi_{\kappa}^{\ell} \Bigl(x^{1}(t), x^{\epsilon_2, 2}(t), \ldots, x^{\epsilon_{\ell}, \ell}(t)\Bigr) \Bigr\Vert_{\infty}  \exp\bigl( \lambda_{\kappa}^{\epsilon,\ell}\, t\bigr)  \mathbb{P}_{x_{\widehat{1,\ell}}}^{\epsilon}\Bigl\{\tau_{\delta}^{\ell,\epsilon} > t\Bigr\}
\end{align*}
If we take the logarithm and divide both sides by $t$, then, further let $t \rightarrow \infty$, we have
\begin{align}
\lambda_{\kappa}^{\epsilon,\ell} & \ge - \liminf_{t \rightarrow \infty} \frac{1} {t} \log \mathbb{P}_{x_{\widehat{1,\ell}}}^{\epsilon} \bigl\{\tau^{\ell,\epsilon} > t \bigr\} \ge - \limsup_{t \rightarrow \infty} \frac{1} {t} \log \mathbb{P}_{x_{\widehat{1,\ell}}}^{\epsilon} \bigl\{\tau^{\ell,\epsilon} > t \bigr\}.  \label{Eq9}
\end{align}
On the other hand, let $B_k \supset D_{\ell} \cup \partial D_{\ell}$ be an open domain with smooth boundary and let $\tau_B^{\ell,\epsilon}$ be the exit-time for the diffusion process $x^{\epsilon_{\ell}, \ell}(t)$ from the domain $B_k$. Further, let $ \psi_{\kappa, B_k}^{\ell} $ and $\lambda_{\kappa, B_k}^{\epsilon,\ell}$ be the principal eigenfunction-eigenvalue pairs for the eigenvalue problem of $\mathcal{L}_{\kappa}^{\epsilon,\ell}$ on $B_k$, with $\psi_{\kappa, B_k}^{\ell} \Bigl(\hat{x}^{1}, \hat{x}^2, \ldots, \hat{x}^{\ell}(t)\Bigr) = 1$, for $\bigl(\hat{x}^{1}, \hat{x}^2, \ldots, \hat{x}^{\ell}\bigr) \in D_{\ell}$.

Then, we have the following
\begin{align*}
& \psi_{\kappa, B_k}^{\ell} \Bigl(\hat{x}^{1}, \hat{x}^2, \ldots, \hat{x}^{\ell}\Bigr) \\
& \quad\quad\quad\quad = \mathbb{E}_{x_{\widehat{1,\ell}}}^{\epsilon} \Biggl\{\exp\bigl( \lambda_{\kappa, B_k}^{\epsilon,\ell} t \bigr) \psi_{\kappa, B_k}^{\ell} \Bigl(x^{1}(t), x^{\epsilon, 2}(t), \ldots, x^{\epsilon_{\ell}, \ell}(t)\Bigr) \mathbf{1}\Bigl\{\tau_B^{\ell,\epsilon} > t\Bigr\} \Biggr\} \\
&\quad\quad\quad\quad \ge \mathbb{E}_{x_{\widehat{1,\ell}}}^{\epsilon} \Biggl\{\exp\bigl( \lambda_{\kappa, B_k}^{\epsilon,\ell} t \bigr) \,\psi_{\kappa, B_k}^{\ell} \Bigl(x^{1}(t), x^{\epsilon_2, 2}(t), \ldots, x^{\epsilon_{\ell}, \ell}(t)\Bigr) \mathbf{1}\Bigl\{\tau_B^{\ell,\epsilon} > t\Bigr\} \Biggr\} \\
  &\quad\quad\quad\quad \ge  \inf_{\bigl(x^{1}, x^{\epsilon_2, 2}, \ldots, x^{\epsilon_{\ell}, \ell}(t)\bigr) \in D_{\ell}}\Bigl\vert \psi_{\kappa, B_k}^{\ell} \Bigl(x^{1}(t), x^{\epsilon, 2}(t), \ldots, x^{\epsilon_{\ell}, \ell}(t)\Bigr) \Bigr\vert   \\
  & \quad\quad\quad\quad\quad\quad\quad\quad \times  \exp\bigl( \lambda_{\kappa, B_k}^{\epsilon,\ell}\, t\bigr)  \mathbb{P}_{x_{\widehat{1,\ell}}}^{\epsilon}\Bigl\{\tau_B^{\ell,\epsilon} > t\Bigr\}.
\end{align*}
Hence, from Equation~\eqref{Eq9}, we have the following 
\begin{align*}
\lambda_{\kappa, B_k}^{\epsilon,\ell} & \le - \limsup_{t \rightarrow \infty} \frac{1} {t} \log \mathbb{P}_{x_{\widehat{1,\ell}}}^{\epsilon} \bigl\{\tau^{\ell,\epsilon} > t \bigr\} \le - \liminf_{t \rightarrow \infty} \frac{1} {t} \log \mathbb{P}_{x_{\widehat{1,\ell}}}^{\epsilon} \bigl\{\tau^{\ell,\epsilon} > t \bigr\}.
\end{align*}
Then, using Proposition~4.10 of \cite{QuaSi08}, we have $\lambda_{\kappa, B_k}^{\epsilon,\ell} \rightarrow \lambda_{\kappa}^{\epsilon,\ell}$ as $B_k \rightarrow D_{\ell} \cup \partial D_{\ell}$. This completes the proof of Proposition~\ref{P1}.
\end{proof}

\subsection{Connection with optimal control problems} \label{S2(2)}
For $\ell = 2, 3, \ldots, n$, define the following HJB equations
\begin{align}
& \mathcal{L}^{\epsilon,\ell} \bigl(\cdot\bigr) \bigl(x^{1}, \ldots, x^{\ell}, u_{\ell}\bigr) = \sum\nolimits_{i=1}^{\ell-1} \Bigl \langle \bigtriangledown(\cdot), \hat{m}_{i\setminus\ell} \bigl(x^{1}, \ldots, x^{i}\bigr) \Bigr\rangle + \Bigl \langle \bigtriangledown(\cdot), m_{\ell}\bigl(x^{1}, \ldots, x^{\ell}, u_{\ell}\bigr) \Bigr\rangle \notag \\ 
& \quad \quad \quad \quad \quad + \frac{1}{2} \operatorname{tr}\Bigl \{\sigma(x^1)\sigma^T(x^1)\bigtriangledown^2(\cdot) \Bigr\} + \sum\nolimits_{j=2}^{\ell} \frac{\epsilon_j}{2} \operatorname{tr}\Bigl \{ I_{d \times d}\bigtriangledown^2(\cdot) \Bigr\}, \label{Eq10}
\end{align}
where
\begin{align*}
\hat{m}_{i\setminus\ell}\bigl(x^{1}(t), \ldots, x^{i}(t)\bigr) = m_{i}\bigl(x^{1}(t), \ldots, x^{i}(t), u_{i}^{\ast}(t)\bigr), \,\, i = 1, 2, \ldots, \ell-1.
\end{align*}
Note that we can associate the above HJB equations with the following family of optimal control problems
\begin{align}
\max_{u_{\ell} \in \mathcal{U}_{\ell}} \Biggl \{ \mathcal{L}^{\epsilon,\ell} \psi^{\ell} \bigl(x^{1}, \ldots, x^{\ell}, u_{\ell}\bigr) + \lambda^{\epsilon,\ell} \psi^{\ell} \bigl(x^{1}, \ldots, x^{\ell}\bigr) \Biggr\},  \label{Eq11}
\end{align}
for $\ell = 2, 3, \ldots, n$.

\begin{proposition} \label{P2}
There exist a unique $\lambda_{\ast}^{\epsilon,\ell} > 0$ (which is the minimum exit rate corresponding to the $\ell$th-subsystem) and $\psi_{\ast}^{\ell} \in C^2(\Omega_{\ell}) \cap C(\bar{\Omega}_{\ell})$, with $\psi_{\ast}^{\ell} > 0$ on $\Omega_{\ell}$, that satisfy the optimal control problem in Equation~\eqref{Eq11}.

Moreover, the admissible Markov control $\kappa_{\ell}^{\ast}$ is optimal if and only if $\kappa_i^{\ast}$ is a measurable selector for 
\begin{align}
\argmax \Bigl\{\mathcal{L}^{\epsilon, i} \psi_{\ast}^{i} \bigl(x^{1}, \ldots, x^{\epsilon_i, i}, \,\cdot \,\bigr) \Bigr\}, \,\, \bigl(x^{1}, \ldots, x^{\epsilon_i, i}\bigr) \in D_1 \times \cdots \times D_i, \label{Eq12}
\end{align}
for $i = 1, 2, \dots, \ell-1$.
\end{proposition}

\begin{proof}
The first claim for $\psi_{\ast}^{\ell} \in W_{loc}^{2,p} \bigl(\Omega_{\ell}\bigr) \cap C\bigl(\bar{\Omega}_{\ell}\bigr)$, with $p>2$, follows from Equation~\eqref{Eq7} (cf. \cite[Theorems~1.1, 1.2 and 1.4]{QuaSi08}). Note that if $\kappa_i^{\ast}$, for $i=1,2, \ldots, \ell$, are measurable selectors of $\argmax \Bigl\{\mathcal{L}^{\epsilon, i} \psi_{\ast}^{i} \bigl(x^{1}, \ldots, x^{\epsilon_i, i}, \,\cdot \,\bigr) \Bigr\}$, with $\bigl(x^{1}, \ldots, x^{\epsilon_i, i}\bigr) \in D_1 \times \cdots \times D_i$. Then, by the uniqueness claim for eigenvalue problem in Equation~\eqref{Eq7}, we have
\begin{align*}
\lambda_{\kappa^{\ast}}^{\epsilon,\ell} = - \limsup_{t \rightarrow \infty} \frac{1} {t} \log \mathbb{P}_{x_{\widehat{1,\ell}}}^{\epsilon} \bigl\{\tau^{\ell, \epsilon} > t \bigr\}, 
\end{align*}
where the the probability $\mathbb{P}_{x_{\widehat{1,\ell}}}^{\epsilon} \bigl\{\cdot\}$ is conditioned with respect to $\bigl(\kappa_1^{\ast}, \kappa_2^{\ast}, \ldots, \kappa_{\ell}^{\ast}\bigr)$. Then, for any other admissible controls $u_i$, for $i=1, 2, \ldots, \ell$, we have
\begin{align*}
\mathcal{L}^{\epsilon, i} \psi_{\ast}^{j} \bigl(x^{1}, \ldots, x^{\epsilon_i, i}, u_i \bigr) + \lambda_{\kappa^{\ast}}^{\epsilon,\ell} \psi_{\ast}^{i} \bigl(x^{1}, \ldots, x^{\epsilon_i, i}\bigr) \le 0, \quad \forall t \ge 0.
\end{align*}
Let $Q \subset \mathbb{R}^d$ be a smooth bounded open domain containing $D_{\ell} \cup \partial D_{\ell}$. Let $\hat{\psi}^{\ell}$ and  $\hat{\lambda}^{\epsilon,\ell}$ be the principal eigenfunction-eigenvalue pairs for the eigenvalue problem of $\mathcal{L}^{\epsilon,\ell}$ on $\partial Q$. 

Let 
\begin{align*}
\hat{\tau}^{\ell,\epsilon} = \inf \Bigl\{ t > 0 \, \bigl\vert \, x^{\epsilon_{\ell}, \ell}(t) \notin Q \Bigr\}, \quad \ell =2, 3, \ldots, n.
\end{align*}
Then, under $u_i$, for $i=1,2, \ldots, \ell$, we have  
\begin{align*}
& \hat{\psi}^{\ell} \Bigl(\hat{x}^{1}, \hat{x}^2, \ldots, \hat{x}^{\ell}\Bigr) \\
&\quad\quad\quad\quad \ge \mathbb{E}_{x_{\widehat{1,\ell}}}^{\epsilon} \Biggl\{\exp\bigl( \hat{\lambda}^{\epsilon,\ell} t \bigr) \,\hat{\psi}^{\ell} \Bigl(x^{1}(t), x^{\epsilon_2, 2}(t), \ldots, x^{\epsilon_{\ell}, \ell}(t)\Bigr) \mathbf{1}\Bigl\{\hat{\tau}^{\ell,\epsilon} > t\Bigr\} \Biggr\} \\
  &\quad\quad\quad\quad \ge  \inf_{\bigl(x^{1}, x^{\epsilon_2, 2}, \ldots, x^{\epsilon_{\ell}, \ell}(t)\bigr) \in D_{\ell}}\Bigl\vert \hat{\psi}^{\ell} \Bigl(x^{1}(t), x^{\epsilon, 2}(t), \ldots, x^{\epsilon_{\ell}, \ell}(t)\Bigr) \Bigr\vert   \\
  & \quad\quad\quad\quad\quad\quad\quad\quad \times  \exp\bigl(\hat{\lambda}^{\epsilon,\ell}\, t\bigr) \mathbb{P}_{x_{\widehat{1,\ell}}}^{\epsilon}\Bigl\{\hat{\tau}^{\ell,\epsilon} > t\Bigr\}.
\end{align*}
Leading to
\begin{align*}
\hat{\lambda}^{\epsilon,\ell} \le - \limsup_{t \rightarrow \infty} \frac{1} {t} \log \mathbb{P}_{x_{\widehat{1,\ell}}}^{\epsilon} \bigl\{\tau^{\ell,\epsilon} > t \bigr\}.
\end{align*}
Letting $Q$ shrink to $D_{\ell}$ and using Proposition~4.10 of \cite{QuaSi08}, then we have $\hat{\lambda}^{\epsilon,\ell} \rightarrow \lambda_{\kappa^{\ast}}^{\epsilon,\ell}$.Thus, we have
\begin{align*}
\lambda_{\kappa^{\ast}}^{\epsilon,\ell} \le - \limsup_{t \rightarrow \infty} \frac{1} {t} \log \mathbb{P}_{x_{\widehat{1,\ell}}}^{\epsilon} \bigl\{\tau^{\ell,\epsilon} > t \bigr\}.
\end{align*}
Combining with Equation~\eqref{Eq5}, this establishes the optimality of $\kappa_{\ell}^{\ast}$ and the fact that $\lambda_{\kappa^{\ast}}^{\epsilon,\ell}$ is the minimum exit rate.

Conversely, let $\bigl(\hat{\kappa}_1^{\ast}, \hat{\kappa}_2^{\ast}, \ldots, \hat{\kappa}_{\ell}^{\ast}\bigr)$ be optimal Markov controls. Then, we have
\begin{align*}
\mathcal{L}^{\epsilon, i} \hat{\psi}^{i} \bigl(x^{1}, \ldots, x^{\epsilon_i, i}, \hat{\kappa}_{i}^{\ast}(x^{\epsilon_i, i}) \bigr) + \hat{\lambda}_{\hat{\kappa}^{\ast}}^{\epsilon,i} \, \hat{\psi}^{i} \bigl(x^{1}, \ldots, x^{\epsilon_i, i} \bigr) = 0
\end{align*}
and
\begin{align*}
\mathcal{L}^{\epsilon, i} \psi_{\ast}^{i} \bigl(x^{1}, \ldots, x^{\epsilon_i, i}, \hat{\kappa}_{i}^{\ast}(x^{\epsilon_i, i}) \bigr) + \lambda_{\kappa^{\ast}}^{\epsilon,i} \, \psi_{\ast}^{i} \bigl(x^{1}, \ldots, x^{\epsilon_i, i} \bigr) \le 0, \quad \forall t > 0,
\end{align*}
with $\hat{\lambda}_{\hat{\kappa}^{\ast}}^{\epsilon,i} = \lambda_{\kappa^{\ast}}^{\epsilon,i}$, for $i = 1, 2, \ldots, \ell$.

Furthermore, notice that $\psi_{\ast}^{j}$ is a scalar multiple of $\hat{\psi}^{j}$ and, at $\bigl(\hat{x}^{1}, \hat{x}^2, \ldots, \hat{x}^{\ell}\bigr) \in \Omega_{\ell}$ (cf. \cite[Theorem~1.4(a)]{QuaSi08}). Then, we see that $\hat{\kappa}_{\ell}^{\ast}$ is a maximizing measurable selector in Equation~\eqref{Eq11}. This completes the proof of Proposition~\ref{P2}.
\end{proof}

\subsection{Exit probabilities} \label{S2(3)} 
In this subsection, for a fixed (given) $T > 0$ and $\ell \in \{2, 3, \ldots, n\}$, we relate the attainable exit probability with which the diffusion process $x^{\ell}(t)$ exits from the domain $D_{\ell}$, i.e.,
\begin{align*}
q^{\ell} \bigl(x^{1}, x^{2}, \ldots, x^{\ell}\bigr) = \mathbb{P}_{x_{\widehat{1,\ell}}}\bigl\{ \tau_{\kappa^{\ast}}^{\ell} \le T \bigr\},
\end{align*}
and a family of smooth solutions for Dirichlet problems on $\Omega_{\ell}$ that are associated with the nondegenerate diffusion process $x^{\epsilon, \ell}(t)$ (see Equation~\eqref{Eq13}) as the limiting case, when $\epsilon_{\ell} \rightarrow 0$, for $\ell = 2,3, \dots, n$.
\begin{remark} \label{R4}
Note that $\tau_{\kappa^{\ast}}^{\ell}$ is the exit-time for the diffusion process $x^{\ell}(t)$ from the domain $D_{\ell}$, with respect to the admissible optimal Markov controls $\kappa_i^{\ast}$, for $i=1,2, \ldots, \ell$.
\end{remark}

\begin{proposition} \label{P3}
Suppose that a set of admissible optimal Markov controls $\kappa_i^{\ast}$, for $i=1,2, \ldots, n$, satisfies Proposition~\ref{P2}. Let $\psi \bigl(x^{1}, x^{2}, \ldots, x^{\ell} \bigr)$, for $t \in (0,\, T)$, be a continuous function on $\partial D_{\ell}$. Then, the attainable exit probability $q^{\ell} \bigl(x^{1}, x^{2}, \ldots, x^{\ell}\bigr)$ with which the diffusion process $x^{\ell}(t)$ exits from the domain $D_{\ell}$ is a smooth solution to the following Dirichlet problem
\begin{align}
\left.\begin{array}{c}
\mathcal{L}^{\ell} \upsilon^{\ell} \bigl(x^{1}, x^{2}, \ldots, x^{\ell} \bigr) = 0 \quad \text{in} \quad \Omega_{\ell}  \\
 \upsilon^{\ell} \bigl(x^{1}, x^{2}, \ldots, x^{\ell} \bigr) = \mathbb{E}_{x_{\widehat{1,\ell}}} \bigl\{\psi^{\ell} \bigl(x^{1}, x^{2}, \ldots, x^{\ell} \bigr)\bigr\} \quad \text{on} \quad  \partial \Omega_{\ell}
\end{array}\right\}.  \label{Eq13}
\end{align}
Moreover, it is a continuous function on $\Omega_{\ell}$ for each $\ell = 2, 3, \ldots, n$.
\end{proposition}

In order to prove the above proposition, we will consider again the nondegenerate diffusion process $(x^{1}(t), x^{\epsilon_2, 2}(t), \ldots, x^{\epsilon_{\ell}, \ell}(t))$, for $\ell = 2, 3, \ldots, n$ (see Equation~\eqref{Eq5}). Then, we relate the attainable exit probability of this diffusion process with that of the Dirichlet problems in Equation~\eqref{Eq13} in the limiting case, when $\epsilon_j \rightarrow 0$, for $j = 2,3, \dots, \ell$. 

Let us define the following notations that will be useful in the sequel.
\begin{align*}
\begin{array}{c}
    \theta = \tau_{\kappa^{\ast}}^{\ell} \wedge T, \quad\quad  \theta^{\epsilon} = \tau_{\kappa^{\ast}}^{\epsilon,\ell} \wedge T,\\
   \bigl\Vert x^{\epsilon_\ell, \ell} - x^{\ell} \bigr\Vert_t = \sup\limits_{0 \le r \le t} \Bigl\vert x^{\epsilon_\ell, \ell}(r) - x^{\ell}(r) \Bigr\vert,\quad \ell = 2, 3, \ldots, n,
\end{array}
\end{align*}
where $\tau_{\kappa^{\ast}}^{\epsilon,\ell}$ is the exit-time for the diffusion process $x^{\epsilon_\ell, \ell}(t)$ from the domain $D_{\ell}$ (cf. $\tau_{\kappa^{\ast}}^{\ell}$ in Remark~\ref{R4}).

Then, we need the following lemma, which is useful for proving the above proposition.

\begin{lemma} \label{L1}
For any initial point $(x_0^{1}, x_0^{2}, \ldots, x_0^{\ell}) \in \Omega_{\ell}$, the following statements hold true
\begin{enumerate} [(i)]
\item $\bigl\Vert x^{\epsilon_\ell, \ell} - x^{\ell} \bigr\Vert_t  \rightarrow 0$, 
\item $\theta^{\epsilon} \rightarrow \theta$, and 
\item $\bigl(x^{\epsilon_2, 2}(\theta^{\epsilon}), \ldots, x^{\epsilon_\ell, \ell}(\theta^{\epsilon})\bigr) \rightarrow \bigl(x^{2}(\theta), \ldots, x^{\ell}(\theta)\bigr)$, 
\end{enumerate}
almost surely, as $\epsilon_\ell \rightarrow 0$, for $\ell = 2, 3, \ldots, n$.
\end{lemma}

\begin{proof} 
Part~(i): Note that, for any $\ell \in \{ 2, 3, \ldots, n\}$, the following inequality holds 
\begin{align*}
\bigl\vert x^{\epsilon_\ell, \ell}(r) - x^{\ell}(r) \bigr\vert &\le  \int_{0}^t \Bigl \vert \hat{f}_{\ell}\bigl(x^{1}(r), x^{\epsilon_2, 2}(r) \ldots, x^{\epsilon_{\ell}, \ell}(r)\bigr) - \hat{f}_{\ell}\bigl(x^{1}(r), x^{2}(r), \ldots, x^{\ell}(r)\bigr) \Bigr\vert dr  \\
& \quad \quad\quad\quad\quad\quad + \sqrt{\epsilon_{\ell}} \bigl \vert W_{\ell}(t) \bigr\vert,  \\
& \le C \int_{0}^t \bigl \vert x^{\epsilon_\ell, \ell}(r) - x^{\ell}(r) \bigr\vert dt + \sqrt{\epsilon_{\ell}} \bigl \vert W_{\ell}(t) \bigr\vert,
\end{align*}
such that
\begin{align*}
\bigl \Vert x^{\epsilon_\ell, \ell}(r) - x^{\ell}(r) \bigr \Vert_t \le C \int_{0}^t \bigl \vert x^{\epsilon_\ell, \ell}(r) - x^{\ell}(r) \bigr\vert dt + \sqrt{\epsilon_{\ell}} \vert W_{\ell}(t)\vert,
\end{align*}
where $C$ is a Lipschitz constant. Using the Gronwall-Bellman inequality, we obtain the following
\begin{align*}
\bigl \Vert x^{\epsilon_\ell, \ell}(r) - x^{\ell}(r) \bigr \Vert_t \le c \sqrt{\epsilon_{\ell}} \bigl\Vert W_{\ell} \bigr\Vert_t,
\end{align*}
where $c$ is a constant that depends on $C$ and $t$. Hence, we have 
\begin{align*}
\bigl \Vert x^{\epsilon_\ell, \ell}(r) - x^{\ell}(r) \bigr \Vert_t  \rightarrow 0 \quad \text{as} \quad \epsilon_\ell \rightarrow 0,
\end{align*}
for each $\ell = 2, 3, \ldots, n$.

Part~(ii): Next, let us show $\theta$ satisfies the following bounds
\begin{align*}
\theta^{\ast} \le \theta \le \theta_{\ast},
\end{align*}
almost surely, where
\begin{align*}
\theta^{\ast} = \min_{\ell \in \{2, 3, \ldots, n\}} \Bigl\{\limsup_{\epsilon_\ell \rightarrow 0} \theta^{\epsilon}\Bigr\}
\end{align*}
and
\begin{align*}
\theta_{\ast} =\max_{\ell \in \{2, 3, \ldots, n\}} \Bigl\{\liminf_{\epsilon_\ell \rightarrow 0} \theta^{\epsilon}\Bigl\}.
\end{align*}
Notice that $D_{\ell}$ is open, then it follows from Part~(i) that if $\theta = \tau^{\ell} \wedge T=T$, then $\theta^{\epsilon} = T$, almost surely, for all $\epsilon_{\ell}$ sufficiently small. Then, we will get Part~(ii). Similarly, if $\theta^{\epsilon} = T$ and $x^{\epsilon, \ell}(\theta^{\epsilon}) \in D_{\ell}$, then the statement in Part~(i) implies Part~(ii). Then, we can assume that $x^{\ell}(\theta) \in \partial D_{\ell}$ and $x^{\epsilon, \ell}(\theta^{\epsilon}) \in \partial D_{\ell}$. Moreover, if $x^{\epsilon, \ell}(\theta^{\epsilon}) \in \partial D_{\ell}$, then, from Part~(i), $x^{\ell}(\theta_{\ast}) \in \partial D_{\ell}$, almost surely, and, consequently, $\theta_{\ast} \ge \theta$, almost surely.

For the case $\theta^{\ast} \le \theta$, let us define an event $\Psi_{a,\alpha}$ (with $a >0$ and $\alpha > 0$) as follow: there exists $t \in [\theta, \theta + a]$ such that the distance $\varrho \bigl(x^{\ell}(t), \partial D_{\ell} \bigr) \ge \alpha$. Notice that if this holds together with $\bigl\Vert x^{\epsilon_\ell, \ell} - x^{\ell} \bigr\Vert_t < \alpha$, then we have $\theta^{\epsilon} < \theta + a$. Hence, from Part~(i), we have $\theta^{\ast} < \theta + a$ on $\Psi_{a,\alpha}$, almost surely.

On the other hand, from Assumption~\ref{AS1}(b), we have the following
\begin{align*}
\mathbb{P}_{x_{\widehat{1,\ell}}} \Bigl\{ \bigcup\nolimits_{\alpha >0} \Psi_{a,\alpha}\Bigr\} = 1.
\end{align*}
Then, we have
\begin{align*}
\mathbb{P}_{x_{\widehat{1,\ell}}} \Bigl\{ \theta^{\ast} < \theta + a\Bigr\} = 1,
\end{align*}
since $a$ is an arbitrary, we obtain $\theta^{\ast} < \theta$, almost surely.

Finally, notice that the statement in Part~(iii) is a consequence of Part~(i) and Part~(ii). This completes the proof of Lemma~\ref{L1}.
\end{proof}
\begin{proof} [Proof of Proposition~\ref{P3}]
Note that, from Assumption~\ref{AS1}(b), it is sufficient to show that $q^{\ell} (x^{1}, \ldots, x^{\ell})$ is a smooth solution (almost everywhere in $\Omega_{\ell}$ with respect to Lebesgue measure) to the Dirichlet problems in Equation~\eqref{Eq13}.

For a fixed $\ell \in \{ 2, 3, \ldots, n\}$, consider the following infinitesimal generator which corresponds to the nondegenerate diffusion processes $(x^{1}(t), x^{\epsilon_2, 2}(t), \ldots, x^{\epsilon_\ell, \ell}(t))$
\begin{align}
 \mathcal{L}^{\epsilon,\ell} \upsilon^{\ell} + \sum\nolimits_{j=2}^\ell \frac{\epsilon_j}{2} \triangle_{x^{\epsilon_j, j}} \upsilon^{\ell}= 0 \quad \text{in} \quad \Omega_{\ell},  \label{Eq14}
\end{align}
where $\triangle_{x^{\epsilon_j, j}}$ is the Laplace operator in the variable $x^{\epsilon_j, j}$ and $\mathcal{L}^{\epsilon, \ell}$ is the infinitesimal generator in Equation~\eqref{Eq5}.

Note that the infinitesimal generator $\mathcal{L}^{\epsilon,\ell}$ is a uniformly parabolic equation and, from Assumption~\ref{AS1}(b), its solution satisfies the following boundary condition
\begin{align}
 \upsilon^{\ell} \bigl(x^{1}, x^{\epsilon_2, 2}, \ldots, x^{\epsilon_\ell, \ell} \bigr) = \psi \bigl(x^{1}, x^{\epsilon_2, 2}, \ldots, x^{\epsilon_\ell, \ell} \bigr) \quad \text{on} \quad \partial D,  \label{Eq15}
\end{align}
where
\begin{align}
 \upsilon^{\ell} \bigl(x^{1},x^{\epsilon_2, 2}, \ldots, x^{\epsilon_\ell, \ell} \bigr) = \mathbb{E}_{x_{\widehat{1,\ell}}}^{\epsilon} \Bigl\{ \psi \bigl(x^{1}(\theta^{\epsilon}), x^{\epsilon_2, 2}(\theta^{\epsilon}), \ldots, x^{\epsilon_\ell, \ell}(\theta^{\epsilon}) \bigr) \Bigr\},  \label{Eq16}
\end{align}
with $\theta^{\epsilon} = \tau^{\epsilon,\ell} \wedge T$.

In particular, let $\psi_k$, with $k=1, 2,\ldots$, be a sequence of bounded functions that are continuous on $\partial D$ and satisfying the following conditions
\begin{align*}
 \psi_k \bigl(x^{1}, x^{\epsilon_2, 2}, \ldots, x^{\epsilon_\ell, \ell} \bigr) = \left\{\begin{array}{l l}
1  &\text{if} \,\, \bigl(x^{1}(t), x^{\epsilon_2, 2}(t), \ldots, x^{\epsilon_{\ell}, \ell}(t) \bigr) \in \Gamma_{\ell}\\
0 &\text{if} \,\,  \bigl(x^{1}(t), x^{\epsilon_2, 2}(t), \ldots, x^{\epsilon_{\ell}, \ell}(t) \bigr) \in \Omega_{\ell}  \\
&  \quad \quad\quad\quad\quad\quad\quad\quad\quad \text{and} \quad \varrho\bigl(x^{\epsilon, \ell}, \partial D_{\ell} \bigr) > \frac{1}{k}
\end{array}\right. 
\end{align*}
and
\begin{align*}
0 \le \psi_k \bigl(x^{1}, x^{\epsilon_2, 2}, \ldots, x^{\epsilon_\ell, \ell} \bigr) \le 1 \,\, & \text{if} \,\,  \bigl(x^{1}(t), x^{\epsilon_2, 2}(t), \ldots, x^{\epsilon_{\ell}, \ell}(t) \bigr) \in \Omega_{\ell} \\
&  \quad \quad\quad\quad\quad\quad\quad\quad \text{and} \quad  \varrho\bigl(x^{\epsilon, \ell}, \partial D_{\ell} \bigr) \le \tfrac{1}{k}.
\end{align*}
Moreover, the bounded functions further satisfy the following
\begin{align}
\bigl\vert \psi_k - \psi_l \bigr\vert \rightarrow 0 \quad \text{as} \quad k,l \rightarrow \infty  \label{Eq17}
\end{align}
uniformly on any compact subset of $\bar{\Omega}_{\ell}$. Then, with $\psi = \psi_k$,
\begin{align*}
 \upsilon_k^{\ell} \bigl(x^{1}, x^{\epsilon_2, 2}, \ldots, x^{\epsilon_\ell, \ell} \bigr) = \mathbb{E}_{x_{\widehat{1,\ell}}}^{\epsilon} \Bigl\{ \psi_k \bigl(x^{1}(\theta^{\epsilon}), x^{\epsilon_2, 2}(\theta^{\epsilon}), \ldots, x^{\epsilon, \ell}(\theta^{\epsilon}) \bigr) \Bigr\}
\end{align*}
satisfies Equations~\eqref{Eq15} and \eqref{Eq16}. Then, from the continuity of $\psi_k$ (cf. Lemma~\ref{L1}, Parts~(i)-(iii)) and the Lebesgue's dominated convergence theorem (see \cite[Chapter~4]{Roy88}), we have the following
\begin{align}
 \upsilon_k^{\ell} \bigl(x^{1}, x^{\epsilon_2, 2}, \ldots, x^{\epsilon_\ell, \ell} \bigr) \rightarrow \underbrace{\mathbb{E}_{x_{\widehat{1,\ell}}}^{\epsilon} \Bigl\{ \psi_k \bigl(x^{1}(\theta^{\epsilon}), x^{\epsilon_2, 2}(\theta^{\epsilon}), \ldots, x^{\epsilon, \ell}(\theta^{\epsilon}) \bigr) \Bigr\}}_
{\triangleq  q_k^{\ell} \bigl(x^{1}, \ldots, x^{\ell} \bigr)}, \label{Eq18}
\end{align}
where $\theta \rightarrow \theta^{\epsilon}$ as $\epsilon_j \rightarrow 0$ for all $j \in \{2, 3, \ldots, \ell\}$. Furthermore, in the above equation, $\bigl(x^{1}(t), x^{2}(t), \ldots, x^{\ell}(t) \bigr)$ is a solution to Equation~\eqref{Eq13}, when $\epsilon_j=0$, for $j=2, 3, \ldots, \ell$, with an initial condition $\bigl(x^{\epsilon,1}(0), x^{\epsilon_2, 2}(0), \ldots, x^{\epsilon_{\ell}, \ell}(0)\bigr) = \bigl(x_0^{1}, x_0^{2}, \ldots, x_0^{\ell}\bigr)$.

Note that $q_k^{\ell} \bigl(x^{1}, \ldots, x^{\ell} \bigr)$ satisfies Equation~\eqref{Eq14}, with $\upsilon^{\ell} =  \upsilon_k^{\ell}$, and, in addition, it is a distribution solution to the Dirichlet problem in Equation~\eqref{Eq13}, i.e.,
\begin{align*}
\int_{\Omega_{\ell}} {\mathcal{L}^{\ell}}^{\ast} \phi q_k^{\ell} d \Omega_{\ell}, &= \lim_{\epsilon_j \rightarrow 0,\,\forall j \in \{2, 3, \ldots, \ell\}}\int_{\Omega_{\ell}} \Bigl({\mathcal{L}^{\ell}}^{\ast} \phi + \sum\nolimits_{j=2}^\ell \frac{\epsilon_j}{2} \triangle_{x^{\epsilon_j, j}} \phi \Bigr) \upsilon_k^{\ell} d \Omega_{\ell}, \\
  &=0, 
\end{align*}
for any test function $\phi \in C_0^{\infty}(\Omega_{\ell})$. 

Finally, notice that 
\begin{align*}
q^{\ell} \bigl(x^{1}, \ldots, x^{\ell} \bigr) = \lim_{k \rightarrow \infty} q_k^{\ell} \bigl(x^{1}, \ldots, x^{\ell} \bigr),
\end{align*}
almost everywhere in $\Omega_{\ell}$. From Assumption~\ref{AS1}(a) (i.e., the hypoellipticity), $q^{\ell} \bigl(x^{1}, x^{2}, \ldots, x^{\ell} \bigr)$ is a smooth solution to Equation~\eqref{Eq13} (almost everywhere) in $\Omega_{\ell}$ and continuous on the boundary of $\Omega_{\ell}$. This completes the proof of Proposition~\ref{P3}.
\end{proof}

\begin{remark} \label{R5}
Here, we remark that the statements in Proposition~\ref{P3} will make sense only if we have the following conditions
\begin{align*}
\tau_{\kappa^{\ast}}^{1}\, \ge\, \tau_{\kappa^{\ast}}^{2}\, \ge \,\cdots \,\ge \, \tau_{\kappa^{\ast}}^{\ell}, 
\end{align*}
where $\tau_{\kappa^{\ast}}^{1} = \inf \bigl\{ t > 0 \, \bigl\vert \, x^{1}(t) \in \partial D_1 \bigr\}$. Moreover, such conditions depend on the constituting subsystems in Equation~\eqref{Eq1}, the admissible controls from the measurable sets $\prod_{i=1}^{\ell} \mathcal{U}_i$, as well as on the given domains $D_{i}$, for $i = 1, 2, \ldots, \ell$.
\end{remark}

%\section*{Acknowledgments}

\end{document}